\documentclass[12pt,leqno]{amsart}

\usepackage[a4paper,margin=25mm]{geometry}

\newtheorem{theorem}{Theorem}
\newtheorem{proposition}[theorem]{Proposition}

\newcommand{\vol}{\operatorname{vol}}

\title{Maximal volumes of pyramids in the unit cube}

\author{Nikolai Nikolov}

\address{N. Nikolov\\Institute of Mathematics and Informatics\\Bulgarian Academy
of Sciences\\Acad. G. Bonchev 8, 1113 Sofia, Bulgaria
\vspace{1mm}
\newline Faculty of Information Sciences\\
State University of Library Studies and Information Technologies\\
Shipchenski prohod 69A, 1574 Sofia, Bulgaria}
\email{nik@math.bas.bg}

\thanks{The author was partially supported by the Bulgarian National Science Fund,
Ministry of Education and Science of Bulgaria under contract KP-06-N82/6.}


\begin{document}

\begin{abstract}
Let $Q_n=[0,1]^n$, $n\ge2$. Denote by $M_n$ be the maximal volume of a pyramid in $Q_n$
whose base is a hyperplane section of $Q_n$, and by $M_n^0$ the corresponding maximum
when the section contains the center of $Q_n$. We prove that $nM_n<(n+1)M_{n+1}^0$.
\end{abstract}

\maketitle

\section{Results}
Let $Q_n=[0,1]^n$, $n\ge2$. Denote by $M_n$ the maximal volume of a pyramid in $Q_n$
whose base is a hyperplane section of $Q_n$, and by $M_n^0$ the corresponding maximum
when the section contains the center $q_n$ of $Q_n$.

By \cite[Theorem~1]{A}, $M_n^0=c_n$, where
$$c_n=\frac{n}{2^n}\sum_{k=0}^{[n/2]}\frac{(-1)^k(n-2k)^{n-1}}{k!(n-k)!}$$
is the volume of the pyramid with vertex $0$ and base $Q_n\cap\{x_1+\cdots+x_n=n/2\}.$

\begin{proposition}\label{p1}
Let $H$ be a hyperplane through $q_n$. A pyramid in $Q_n$ with base $Q_n\cap H$ has volume
$M_n^0$ if and only if $H$ is orthogonal to a main diagonal of $Q_n$ and its vertex is an
endpoint of that diagonal.
\end{proposition}

It is well known that $\sqrt n\,c_n\to\sqrt{3/(2\pi)}$ (see, e.g., \cite{A}). Note that 
$c_1=c_2=1/2$ and $c_3=3/8$. Recall also that $M_2=1/2$.

\begin{proposition}\label{p2}
$M_3=(1+\sqrt2)/6$.
\end{proposition}

\begin{proposition}\label{p3} $c_n<M_n<c_{n-1}$ for $n\ge 3.$
\end{proposition}

Therefore $M_n>M_{n+1}$ for $n\ge 2$, and $\sqrt n\,M_n\to\sqrt{3/(2\pi)}$.

The proof of Proposition \ref{p3} is based on the following upper bound for $M_n$.

\begin{theorem}\label{t4}
$nM_n<(n+1)M_{n+1}^0$ for $n\ge2$.
\end{theorem}

\section{Proofs}

\begin{proof}[Proof of Proposition~\ref{p1}]
It suffices to prove that if a pyramid in $Q_n$ with base $Q_n\cap H$ has
volume $M_n^0$, then $H$ is orthogonal to a main diagonal of $Q_n$, since the
assertion concerning the vertex and the converse are immediate.

For $n=2$, let $v=(a,b)\ne0$ be normal to $H$. The maximal area of a triangle
in $Q_2$ having $Q_2\cap H$ as a side is
$$\frac{|a|+|b|}{4\max\{|a|,|b|\}}\le\frac12=M_2^0,$$
with equality only if $|a|=|b|$.

Assume now that $n\ge3$. Let $\mathbf 1=(1,\ldots,1)\in\mathbb R^n$ and
let $z\ne0$ satisfy $\sum_{i=1}^n z_i=0$. Put
$v(\varepsilon)=\mathbf 1+\varepsilon z$, and let $V_n(\varepsilon)$ denote
the maximal volume of a pyramid in $Q_n$ whose base is the central section
orthogonal to $v(\varepsilon)$. Using the integral representation of 
$V_n(\varepsilon)$ (see, e.g., \cite[(3)]{A}), it follows that $V_n$ is twice 
differentiable at $0$ and $V_n(0)=c_n$, $V_n'(0)=0$, 
$V_n''(0)=\|z\|_2^2\bigl(c_n-c_{n-2}\bigr)<0$.
Hence the diagonal direction $\mathbf 1$ is a strict local maximizer.

Suppose that $v\ne\mathbf 1$ is another equality direction with non-negative
coordinates, normalized by $\sum_{i=1}^n v_i=n$. The proof of
\cite[Theorem~1]{A} shows that, for every $\varepsilon\in[0,1]$,
$(1-\varepsilon)\mathbf 1+\varepsilon v$ is again an equality direction,
contradicting the strict local maximality of $\mathbf 1$. Thus, by symmetry,
equality implies that $H$ is orthogonal to a main diagonal of $Q_n$.
\end{proof}

\begin{proof}[Proof of Proposition~\ref{p2}]
We may consider pyramids in $Q_3^0=[-1/2,1/2]^3$ whose bases are plane
sections of $Q_3^0$. Let $B$ be such a base. Writing the plane of $B$ as
$xX+yY+zZ=t/2$, by symmetry and homogeneity, we may assume $x\ge y\ge z\ge0$,
$t\ge0$, $x+y+z+t=1$. Then a pyramid in $Q_3^0$ with base $B$ has maximal
volume $V/6$, where
$$V=\frac{\operatorname{area}(B)}{\sqrt{x^2+y^2+z^2}}.$$
By \cite[(2.1)]{KK}, $V$ is given by five expressions,
corresponding to a parallelogram, a hexagon, a pentagon, a trapezoid, and a
triangle. Maximizing these expressions under the above normalization gives,
respectively, $2,\ 1+\sqrt2,\ 8-4\sqrt2,\ 9/4,\ 2$. Thus $M_3=(1+\sqrt2)/6$.
\end{proof}

\begin{proof}[Proof of Theorem~\ref{t4}]
For $\|a\|_2=1$ and $|t|\le \|a\|_1/2$, put
$H_{a,t}=\{x\in\mathbb R^n:\langle a,x-q_n\rangle=t\},$
$A_a(t)=\vol_{n-1}(Q_n\cap H_{a,t}).$ Let $C$ be a pyramid in $Q_n$ with 
base $Q_n\cap H_{a,t}$. By symmetry, we may assume
that $t\ge0$. Since the height of $C$ is at most $\|a\|_1/2+t$, we have
$$\vol_n(C)\le\frac{1}{2n}(\|a\|_1+2t)A_a(t).$$

Set $b=(a,2t)\in\mathbb R^{n+1}$ and let
$H_b=\{y\in\mathbb R^{n+1}:\langle b,y-q_{n+1}\rangle=0\}.$
Slicing $B=Q_{n+1}\cap H_b$ by the last coordinate gives
$$\vol_n(B)=\sqrt{1+4t^2}\int_{-1/2}^{1/2}A_a(2ts)\,ds.$$
By the Brunn--Minkowski inequality, $A_a(r)^{1/(n-1)}$ is concave on
its support. As $A_a$ is even, it is non-increasing on $[0,\infty)$. Hence
$$\vol_n(B)\ge\sqrt{1+4t^2}\,A_a(t).$$
Since the maximal height of a pyramid in $Q_{n+1}$ with base $B$ is
$(\|a\|_1+2t)/(2\sqrt{1+4t^2})$, we have
$$M_{n+1}^0\overset{\mathrm{(a)}}{\ge}\frac{1}{n+1}\frac{\|a\|_1+2t}{2\sqrt{1+4t^2}}\,
\vol_n(B)\overset{\mathrm{(b)}}{\ge}\frac{\|a\|_1+2t}{2(n+1)}A_a(t).$$
Combining this with the estimate for $\vol_n(C)$ above yields
$n\vol_n(C)\le(n+1)M_{n+1}^0.$

It remains to exclude equality. If equality held, then both inequalities (a) and (b)
must be equalities. Equality in (b) gives $A_a(r)=A_a(t)$ for $0\le r\le t$. By the
equality case of the Brunn--Minkowski inequality, the sections $Q_n\cap H_{a,0}$ and
$Q_n\cap H_{a,t}$ are translates of one another. On the other hand, equality in (a),
together with Proposition~\ref{p1}, gives that $a$ is a diagonal direction and
$t=1/(2\sqrt n)>0$. Hence the two sections are not translates, a contradiction.
\end{proof}

\begin{proof}[Proof of Proposition~\ref{p3}]
Recall that (see, e.g., \cite[(1.4)]{AG}), for $0<s<n$, the volume of the pyramid
with vertex $0$ and base $Q_n\cap\{x_1+\cdots+x_n=s\}$ is
$$V_n(s)=\frac{s}{n!}\sum_{k=0}^{[s]}(-1)^k\binom nk(s-k)^{n-1}.$$
For $n\ge3$, $V_n$ is differentiable at $n/2$ and $V_n'(n/2)=2c_n/n>0$.
So $M_n>V_n(n/2)=c_n$.

Next, it is known that $(n+1)c_{n+1}<nc_{n-1}$ for every odd $n\ge3$ and for every
$n\ge 137$ (see, e.g., \cite{AG}). Exact integer arithmetic verifies this inequality
for all even integers $4\le n\le136$. Then, by $M_{n+1}^0=c_{n+1}$ and Theorem \ref{t4},
$M_n<c_{n-1}$.
\end{proof}

\end{document}